\pdfoutput=1
\documentclass[11pt]{article}
\usepackage{geometry}                
\usepackage[parfill]{parskip}    
 \usepackage{color} 
\usepackage{graphicx}
\usepackage{amssymb}
\usepackage{amsmath}
\usepackage{amsthm}
\usepackage{epstopdf}

\title{Apolarity for determinants and permanents of generic matrices}
\author{Masoumeh (Sepideh) Shafiei\\ \footnotesize{Department of Mathematics, Northeastern University, 
Boston, MA 02115, USA}}

\date{March 2013}                                           

\newtheorem{thm}{Theorem}[section]
\newtheorem{lem}[thm]{Lemma}
\newtheorem{cor}[thm]{Corollary}
\newtheorem{prop}[thm]{Proposition}
\newtheorem*{ack}{Acknowledgment}

\theoremstyle{definition}
\newtheorem{ex}[thm]{Example}

\theoremstyle{definition}
\newtheorem{defi}[thm]{Definition}

\theoremstyle{definition}
\newtheorem*{Notation}{Notation}

\theoremstyle{definition}
\newtheorem*{rmk}{Remark}

\theoremstyle{definition}
\newtheorem{remark}[thm]{Remark}

\begin{document}
\maketitle


\begin{abstract} 

We show that the apolar ideals to the determinant and permanent of a generic matrix, the Pfaffian of a generic skew symmetric matrix  and the Hafnian of a generic symmetric matrix are each generated in degree two. In each case we specify the generators and a Gr\"{o}bner basis of the apolar ideal. As a consequence, using a result of K.~Ranestad and F.-O. Schreyer we give lower bounds to the cactus rank and rank of each of these invariants. We compare these bounds with those obtained by J. Landsberg and Z. Teitler.

\end{abstract}

\section{Introduction}
\label{intro}

This paper is originally motivated by a question from Zach Teitler about the generating degree of the annihilator ideal of the determinant and the permanent of a generic $n\times n$ matrix. Here annihilator is meant in the sense of the apolar pairing, i.e. Macaulay's inverse system. Our main result is that the apolar ideals of the determinant and of the permanent of a generic matrix  are generated in degree $2$ (Theorems \ref{thm:main-generic-determinant} and \ref{thm:main-generic-permanent}). The reason for Teitler's interest in this problem is the recent paper by Kristian Ranestad and Frank-Olaf~Schreyer \cite{RS}, which gives a lower bound for smoothable rank, border rank and cactus rank of a homogeneous polynomial in terms of the generating degree of the apolar ideal and the dimension of the Artinian apolar algebra defined by the apolar ideal. We apply this and our result to bounding the scheme/cactus length of the determinant and the permanent of the generic matrix (Theorem \ref{generic-det-perm-RS-rank}). In section \ref{Annihilator of the Pfaffian and Hafnian} we give the analogous result for the annihilator ideal of the Pfaffian of a generic skew symmetric matrix  (Theorem \ref{thm:Pfaffian-main-theorem}) and the annihilator of the Hafnian of a generic symmetric matrix (Theorem \ref{thm:Hafnian-main-theorem-cor})

In a sequel paper \cite{Sh2} we study the apolar ideal of the determinant and permanent of the generic symmetric matrix.

Let  $\sf k$ be a field of characteristic zero or characteristic $p>2$, and let $A=(a_{ij})$ be a square matrix of size $n$ with $n^{2}$ distinct variables. The determinant and permanent of $A$ are homogeneous polynomials of degree $n$. Let $R=\sf k$$ [ a_{ij}]$ be a polynomial ring and $S=\sf k$$[d_{ij}]$ be the ring of inverse polynomials associated to $R$, and let $R_k$ and $S_k$ denote the degree-$k$ homogeneous summands. Then $S$ acts on $R$ by contraction:

\begin{equation}
 ({d_{ij}})^k \circ (a_{uv})^\ell=\begin{cases}
  a_{uv}^{\ell-k} &  \text{if $(i,j)=(u,v)$},\\
   0 &  \text{otherwise}.
\end{cases}
\end{equation}
If $h\in S_k$ and $F \in R_n$, then we have $h\circ F\in R_{n-k}$. This action extends multilinearly to the action of $S$ on $R$. When the characteristic of the field $\sf k$ is zero or $\mathrm{char}{\sf k}=p$ greater than the degree of $F$, the contraction action can be replaced by the action of partial differential operators without coefficients (\cite{IK}, Appendix A, and \cite{Ge}). 

\vspace{.2in}
\begin{defi}
 To each degree-$j$ homogeneous element, $F\in R_j$ we associate $I=\mathrm {Ann} (F)$ in $S=\sf k$$[d_{ij}]$ consisting of polynomials $\Phi$ such that $\Phi\circ F=0$. We call $I=\mathrm {Ann} (F)$, the \emph{apolar ideal} of $F$;  and the quotient algebra $S/\mathrm {Ann} (F)$ the apolar algebra of $F$.\par
Let $F \in R$, then $\mathrm {Ann} (F) \subset S$ is an ideal and we have  
 
 \begin{equation*}
{(\mathrm {Ann}(F))}_{k}=\{ h \in S_k|h \circ F=0\}.
 \end{equation*}
 \end{defi}
 \vspace{.2in}
 \begin{remark}
 Let $\phi: (S_i,R_i)\rightarrow \sf k$ be the pairing $\phi(g,f)=g\circ f$,  and $V$ be a vector subspace of $R_k$, then we have
 \begin{equation} \label{eq:Intro-dim}
 \dim_{\sf k} (V^\perp)=\dim_{\sf k} S_k-\dim_{\sf k} V.
 \end{equation}
 \end{remark}

For $V\subset R_k$, we denote by $V^\perp = \mathrm{Ann}(V)\cap S_k$.

Let $F$ be a form of degree $j$ in $R$. We denote by  $<F>_{j-k}$ the vector space $S_k \circ F \subset R_{j-k}$.  (\cite{IK}). 
 
We denote by  $M_k(A)$ the vector subspace of $R$ spanned by the $k \times k$ minors of $A$.
 \vspace{.2in}
\begin{lem}\label{lem:intro-gen}
 \begin{equation}
S_k\circ(\det(A))=M_{n-k}(A) \subset R_{n-k}.
 \end{equation}
 \end{lem}
 \begin{proof}
 It is easy to see that 
  \begin{equation*}
S_k\circ(\det(A))\subset M_{n-k}(A) \subset R_{n-k}.
 \end{equation*}
 For the other inclusion, let $M_{\widehat I,\widehat J}(A), I=(i_1,\ldots ,i_k), J=(j_1,\ldots j_k), 1\le i_1\le i_2\le \cdots \le i_k\le n, 1\le j_1\le j_2\le \cdots \le j_k\le n$ be the $(n-k)\times (n-k)$ minor of $A$ one obtains by deleting the $I$ rows and $J$ columns of $A$. Now it is easy to see that $$M_{\widehat{I},\widehat{J}}=\pm (d_{i_1,j_1}\cdot d_{i_2,j_2}\cdots d_{i_k,j_k}) \circ \det (A).$$ Hence $M_{\widehat{I},\widehat{J}}\in S_k \circ (\det(A))$. 
 
 \end{proof}

\vspace{.2in}
\begin{remark}(see \cite{IK}, page 69, Lemma 2.15)\label{remark:introIK}
Let $F \in R$ and $\deg F=j$ and $k \leq j$. Then we have

 \begin{equation}\label{eq:introIK}
{(\mathrm {Ann}(F))}_{k}=\{h \in S_k| h \circ S^{j-k} F=0\}=(\mathrm {Ann} (S^{j-k} F))_k.
 \end{equation}

\end{remark}
 

\vspace{.2in}
\begin{remark}\label{generalnonsense-intro}
By Lemma 1.3 and Remark 1.4 we have
\begin{equation*}
\mathrm {Ann}(\det(A))_k={M_k(A)}^{\perp}.
\end{equation*}
\end{remark}
\vspace{.2in}

\vspace{.2in}

\begin{ex} Let $n=3$,\begin{center}  $ A= \left(%
\begin{array}{ccc}
 
  a_{11} & a_{12} & a_{13} \\
  a_{21} & a_{22} & a_{23} \\
  a_{31} & a_{32} & a_{33} \\

\end{array}%
\right)$.

\end{center}
Let $P_{ij}$ and $M_{ij}$ be respectively the permanent and the determinant corresponding to the entry $a_{ij}$.
\emph{Question}: Does $P_{11}=d_{22}d_{33}+d_{32}d_{23}$ annihilate $\det(A)= a_{11}M_{11}+a_{12}M_{12}+a_{13}M_{13}$? The following computations allow us to answer this.

$P_{11}\circ a_{11}M_{11}= (d_{22}d_{33}+d_{23}d_{32})\circ (a_{11}a_{22}a_{33}-a_{11}a_{23}a_{32})=a_{11}-a_{11}=0.$

$P_{11}\circ a_{12}M_{12}= (d_{22}d_{33}+d_{23}d_{32})\circ (a_{12}a_{21}a_{33}-a_{12}a_{23}a_{31})=0.$

$P_{11}\circ a_{13}M_{13}= (d_{22}d_{33}+d_{23}d_{32})\circ (a_{13}a_{21}a_{32}-a_{13}a_{22}a_{31})=0 .$

Hence $ P_{11}$ annihilates the determinant. \newline It is easy to see that when $n=3$, $ P_{ij}\circ M_{kl}=0$ for each $1\leq i,j,k,l\leq 3$.  So in the case $n=3$ the annihilator of the determinant of a generic matrix certainly contains all its $2 \times 2$ permanents.
\end{ex}

\vspace{.2in}
\begin{ack} I am deeply grateful to my advisor Prof. Anthony Iarrobino whose help, stimulating ideas and encouragement helped me in working on this problem and writing this paper. I am very thankful to Prof. Zach Teitler for suggesting this problem and his helpful comments and also Prof. Aldo Conca and Prof. Larry Smith for their valuable comments and suggestions. I also gratefully acknowledge support in summer 2012 from the Ling-Ma fellowship.  \par\medskip
\end{ack}


\section{The apolar algebras associated to the $n\times n$ generic matrix}
\label{generic}

In this section we determine the annihilator ideals of the determinant and the permanent of a generic $n\times n$ matrix.  In section \ref{sec:Hilbert-generic} we review the dimension of the subspace of  $k\times k$ minors and permanents of an $n\times n$ generic matrix. In section \ref{Generators of the apolar ideal}  we determine the generators of the apolar ideal to the determinant and permanent of a generic matrix.

We continue to employ the notations of section \ref{intro}, so $R=\sf k$$ [ a_{ij}]$ is a polynomial ring and $S=\sf k$$[d_{ij}]$ is the ring of inverse polynomials associated to $R$, and $S$ acts on $R$ by contraction.

\subsection{Hilbert function and dimension of spaces of minors and permanents}
\label{sec:Hilbert-generic}

Denote by $\mathfrak A_A=S/(\mathrm {Ann} (\det (A))$ the \emph{apolar algebra} of the determinant of the matrix $A$. Recall that the Hilbert function of $\mathfrak A_A$ is defined by $H(\mathfrak A_A)_i=\dim_{\sf{k}} (\mathfrak A_A)_{i}$ for all $i=0,1,\ldots \,.$ 
\vspace{0.2in}
\begin{defi}
\label{def:degree}
Let $F$ be a polynomial in $R$, we define the $\deg(\mathrm {Ann}(F))$ to be the length of $S/Ann(F)$.
\end{defi}

The number of the $k \times k$ minors and permanents of a generic $n\times n$ matrix is ${n\choose k }^2$. The $k\times k$ minors form a linearly independent set (\cite{BC} Theorem 5.3 and Remark 5.4), and the $k\times k$ permanents form another linearly independent set. To show the linearly independence of these two sets we choose a term order, for example the diagonal order where the main diagonal term is a Gr\"{o}bner initial term. Now the initial terms give a basis for the two spaces (\cite{LS}, page 197). So the dimension of the space of $k\times k$ minors of an $n \times n$ matrix and the dimension of the space of $k\times k$ permanents of an $n\times n$ matrix are both ${n\choose k }^2$. By Lemma~\ref{lem:intro-gen} and Remark \ref{generalnonsense-intro} we have
\begin{equation}
\label{eq:Hilbert-generic}
H(S/\mathrm {Ann}(\det{A}))_k=H(S/\mathrm {Ann}(Perm A))_k={n\choose k }^2.
\end{equation}
So the length $\dim_{\sf{k}} (\mathfrak A_A)$ satisfies
\begin{equation}
\label{eq:sum-Hilbert-generic}
\dim_{\sf k}(\mathfrak A_A)=\sum_{k=0}^{k=n} {n\choose k }^2={2n\choose n }.
\end{equation}

A combinatorial proof of the Equation \ref{eq:sum-Hilbert-generic} can be found in \cite{ST}, Example 1.1.17.


\subsection{Generators of the apolar ideal}
\label{Generators of the apolar ideal}

In this section we determine the generators of the apolar ideal of the determinant and permanent of a generic matrix. 
\vspace{0.2in}
\begin{Notation} For a generic $n\times n$ matrix $A=(a_{ij})$, the permanent of $A$ is a polynomial of degree $n$ defined as follows:

\begin{equation*} 
\mathrm{Per}(A)=\sum _ {\sigma \in S_n}  \Pi a_{i,\sigma(i)}
\end{equation*}

\end{Notation}

\begin{lem} \label{lem:generic-1}
Let $A=(a_{ij})$ be a generic $n\times n$ matrix. Then each $2\times 2$ permanent of $D=(d_{ij})$ annihilates the determinant of $A$.
\end{lem}
\begin{proof} Assume we have an arbitrary $2\times 2$ permanent $d_{ij}d_{kl}+d_{il}d_{kj}$ corresponding to \begin{center}  $ P= \left(%
\begin{array}{cc}
  d_{ij} & d_{il}  \\
  d_{kj} & d_{kl} \\
  
\end{array}%
\right)$

\end{center}
Reacll that $\det(A)=\sum _ {\sigma \in S_n} Sgn(\sigma) \Pi a_{i,\sigma(i)}$. There are $n!$ terms in the expansion of the determinant. If a  term does not contain the monomial $a_{ij} a_{kl}$ or the monomial $a_{il} a_{kj}$ then the result of the action of the permanent $d_{ij}d_{kl}+d_{il}d_{kj}$ on it will be zero. There are $(n-2)!$ terms which contain the monomial $a_{ij} a_{kl}$ and $(n-2)!$ terms which contain the monomial $a_{il} a_{kj}$. So assume we have a permutation $\sigma_{1}$ of $n$ objects having $a_{ij}$ and $ a_{kl}$ respectively in it's $i-th$ and $k-th$ place. Corresponding to $\sigma_{1}$ we also have a permutation $\sigma_{2}=\tau \sigma_{1}$, 
 where $\tau=(j,l)$ is a transposition and $sgn (\sigma_{2})=sgn (\tau\sigma_{1})=-sgn( \sigma_{1})$. Thus corresponding to each positive term in the determinant which contains the monomial $a_{ij} a_{kl}$ or the monomial $a_{il} a_{kj}$ we have the same term with the negative sign, thus the resulting action of the permanent $d_{ij}d_{kl}+d_{il}d_{kj}$ on $\det (A)$ is zero. 
\end{proof}
\vspace{.2in}
\begin{defi} \label{def:generic-spaces}
Let $A=(a_{ij})$ and $D=(d_{ij})$ be two generic matrices with entries  in the polynomial ring $R=\sf k$$[a_{ij}]$, and in the ring of differential operators $S=\sf k$$[d_{ij}]$, respectively. Let $\{\mathcal P_A\}$, $\{\mathcal M_A\}$, $\{\mathcal P_D\}$ and $\{\mathcal M_D\}$ denote the set of all $2\times 2$ permanents and the set of all $2 \times 2$ minors of $A$ and $D$, respectively. And let $\mathcal P_A$ , $\mathcal M_A=M_2(A)$, $\mathcal P_D$ and $\mathcal M_D=M_2(D)$ denote the spaces they span, respectively. 
\end{defi}
\vspace{.2in}
\begin{cor}\label{cor:generic2by2}
Each $2\times 2$ permanent of $D$ annihilates $\mathcal M_A$

\end{cor}
\begin{proof} By Lemma  \ref{lem:generic-1}, $P_D\circ \det (A)=0$. Let $F=\det (A)$. 
We have $$ (\mathrm {Ann} F)_2= (\mathrm {Ann}(S_{j-2}\circ F))_2.$$ Hence
 $$\mathcal P_D \circ \det (A)=0 \Longleftrightarrow  \mathcal P_D \circ S_{j-2}(\det(A))=0  \Longleftrightarrow \mathcal P_D\circ \mathcal M_A=0.$$
\end{proof}

We also know that the square of an element, or any product of two or more elements of the same row or column of $D$ annihilates $\det(A)$.

\vspace{.2in}
\begin{defi} \label{def:generic-acceptable} A monomial in the $n^2$ variables of the ring $S=\sf k$$[d_{ij}]$ is \emph{acceptable}, if it is square free and has no two variables from the same row or column of $D$. A polynomial is acceptable if it can be written as the sum of acceptable monomials.
\end{defi}
We denote by $<X>$ the $\sf{k}$-vector space span of the set $X$.

\vspace{.2in}
\begin{lem}\label{lem:generic-acceptables}
 $\mathcal P_D\oplus \mathcal M_D=<$degree 2 acceptable polynomials in $S>$. \end{lem}

\begin{proof} Let $d_{ij}d_{kl}$ be an arbitrary acceptable monomial of degree 2. Since $\mathrm{char} (\sf k)\not=2$ we have:
\begin{equation*}
d_{ij}d_{kl}=1/2((d_{ij}d_{kl}-d_{ij}d_{kl})+(d_{ij}d_{kl}+d_{ij}d_{kl})).
\end{equation*}

By Equation \ref{eq:Hilbert-generic}

\begin{equation*}
\dim \mathcal P_D=\dim \mathcal M_D={n\choose2}^2. 
\end{equation*}

Let $\Psi=<$degree 2 acceptable polynomials in $S>$. Then

\begin{equation*}
\dim\Psi=\dim S_2-\dim \mathcal U_D= {n^2+1\choose 2}-(n^2+{n\choose 2}(2n)).
\end{equation*}

So we have

\begin{equation*}
\dim(\mathcal P_D+ \mathcal M_D)={n^2+1\choose 2}-(n^2+{n\choose 2}(2n))=\dim \mathcal P_D+\dim  \mathcal M_D.
\end{equation*}

Hence $\mathcal P_D\cap \mathcal M_D=0$.

\end{proof}
Denote the space of all unacceptable polynomials of degree 2 by $\mathcal U_D$. We have shown that $(\mathcal P_D+\mathcal U_D)\circ \mathcal M_A=0$ so $\mathcal P_D+\mathcal U_D \subset \mathrm {Ann} (\mathcal M_A) $. Then Equation \ref{eq:Intro-dim} implies $$\dim_{\sf k} (\mathrm {Ann}( \mathcal M_A))_2=\dim_{\sf k} S_2- \dim_{\sf k} \mathcal M_A.$$ 

\begin{lem}\label{lem:generic-pre1}
$ \mathrm {Ann} (\mathcal M_A)\cap S_2 = \mathcal P_D+\mathcal U_D$.
\end{lem}
\begin{proof} 



By Lemma \ref{lem:generic-acceptables} we have $\mathcal P_D+ \mathcal M_D$ is complementary to $\mathcal U_D$. So we have 

$$
\dim ((\mathrm {Ann} (\mathcal M_A))_2)=\dim S_2 -\dim \mathcal M_A=\dim  \mathcal P_D+\mathcal U_D.
$$

\end{proof}
 
\begin{Notation}
We define the homomorphism $\xi:R\rightarrow S$ by setting $\xi(a_{ij})=d_{ij}$; for a monomial $v\in R$ we denote by  $\hat{v}=\xi(v)$ the corresponding monomial of $S$.
\end{Notation}
\vspace{.2in}
\begin{remark}\label{remark:main-ann}
Let $f= \sum_{i=1}^{i=k} \alpha_i v_i  \in R_n$ with $\alpha_{i}\in \sf k$ and with $v_i$'s linearly independent monomials. Then we will have:
\begin{equation}\label{eq:main-ann}
\mathrm {Ann}(f) \cap S_n=<\alpha_j \hat{v_1}-\alpha_1\hat{v_j},<v_1,...,v_k>^\perp>,
\end{equation}
where $<v_1,...,v_k>^\perp=\mathrm {Ann}(<v_1,...,v_k>)\cap S_n$.
\end{remark}
\vspace{.2in}

\begin{lem}\label{lem:main-pre}
\begin{equation}\label{eq:main-pre}
(\mathcal P_D+\mathcal U_D)_k \subset \mathrm {Ann}(M_k(A))  \cap S_k.
 \end{equation}
\end{lem}

\begin{proof}
We have:

(1) $ \mathcal P_D\circ \det (A)=0  \Longleftrightarrow  \mathcal P_D\circ S_{n-2}(\det(A))=0 \Longleftrightarrow \mathcal P_D\circ \mathcal M_A=0.$

(2) ($\mathrm {Ann}(\det(A))) \cap S_2= \mathcal P_D+\mathcal U_D \Rightarrow S_{k-2} (\mathcal P_D+\mathcal U_D)\circ (S_{n-k} \circ \det (A))=0$.

$ \Rightarrow S_{k-2}(\mathcal P_D+\mathcal U_D) \circ M_k(A)=0$.

$ \Rightarrow (\mathcal P_D+\mathcal U_D)_k \circ M_k(A)=0$. (By Remark \ref{remark:introIK})

So Equation \ref{eq:main-pre} holds.

\end{proof}

\begin{prop}\label{prop:generic-n} For a generic $n\times n$ matrix $A$ with $n\geq2$, we have
\begin{equation*}
(\mathcal P_D+\mathcal U_D)_n= \mathrm {Ann}(\det (A))  \cap S_n.
 \end{equation*}
\end{prop}
\begin{proof} Using Equation \ref{eq:main-pre} we only need to show
\begin{equation*}
(\mathcal P_D+\mathcal U_D)_n \supset \mathrm {Ann}(\det (A))  \cap S_n.
\end{equation*}
We use induction on $n$. For $n=2$ the equality is easy to see. Next we verify that the proposition holds for the case $n=3$. We need to see that the space of $2\times 2$ permanents of $D$ generates $ \mathrm {Ann} (\det (A))_3/\mathcal U_D$, i.e., $\mathrm {Ann}(M_3(A))_3/\mathcal U_D$. Corresponding to each term in the determinant, there is a permutation of three objects $\sigma$ such that we can write the term as $a_{1\sigma(1)}a_{2\sigma(2)}a_{3\sigma(3)}$. Consider the degree three binomial $b=a_{1\sigma(1)}a_{2\sigma(2)}a_{3\sigma(3)}-a_{1\tau(1)}a_{2\tau(2)}a_{3\tau(3)}$, where $\tau\neq\sigma$. Without loss of generality we can assume that $\sigma$ is the identity, so we consider the binomial $b=a_{11}a_{22}a_{33}-a_{1\tau(1)}a_{2\tau(2)}a_{3\tau(3)}$. If these two monomials have a common variable i.e., $\tau(i)=i$ for some $i=1,2,3$, then the binomial will be of the form $ b= a_{ii} (a_{jj}a_{kk}-a_{jk}a_{kj})$,  $1\leq
i,j,k,l\leq 3$, so we will have $b=a_{ii}M_{ii}$ and, as we have shown previously, $P_{ii}=d_{jj}d_{kk}-d_{jk}d_{kj}$ annihilates it.  Assume that the monomials $a_{11}a_{22}a_{33}$ and $a_{1\tau(1)}a_{2\tau(2)}a_{3\tau(3)}$ do not have any common factor. We can add and subtract another term $a_{1\beta(1)}a_{2\beta(2)}a_{3\beta(3)}$, where $\beta$ is a permutation, such that it will have one
 common factor with $a_{11}a_{22}a_{33}$ and one common factor with $a_{1\tau(1)}a_{2\tau(2)}a_{3\tau(3)}$. By reindexing we can take $\beta(1)=\tau(1)$, $\beta(2)=2$ and then we can determine $\beta(3)$ according to the other two choices. Then by factorizing we get a binomial of the form $a_{ij}M_{ij}+a_{kl}M_{kl}$, where the first term can be annihilated by the permanent of the matrix $D$ corresponding to $d_{ij}$ and the second term can be annihilated by the permanent of the matrix $D$ corresponding to the element $d_{kl}$. So by Equation \ref{eq:main-ann} we are done. For example, if we have the binomial $a_{11}a_{22}a_{33}-a_{13}a_{21}a_{32}$ we can add and subtract the term $a_{11}a_{23}a_{32}$ which has one common factor with $a_{11}a_{22}a_{33}$ and one common factor with $a_{13}a_{21}a_{32}$ so we will get $a_{11}(a_{22}a_{33}-a_{23}a_{32})+ a_{32}(a_{11}a_{23}-a_{13}a_{21})$ which is $a_{11}M_{11}+a_{32}M_{32}$. And as we have shown before it can be annihilated by the space of $2\times 2$ permanents. So by Equation~\ref{eq:main-ann} we are done.

When $n$ is larger than $3$ then by the induction assumption we can assume that the proposition holds for all $k \leq n-1$. By the Remark \ref{remark:main-ann} it is enough to show that if $b$ is a binomial of the form Equation \ref{eq:main-ann}, in $\mathrm {Ann}(\det (A))  \cap S_n$, then  $b\in (\mathcal P_D+\mathcal U_D)_n$. Assume $b=b_1+b_2$ is of degree $n$. If the two terms, $b_1$ and $b_2$ are monomials in $S$ and have a common factor $l$, i.e., $b_1=la_1$ and $b_2=la_2$, then $b=l(a_1+a_2)$ where $a_1$ and $a_2$ are of degree at most $n-1$. So by  the induction assumption the proposition holds for the binomial $a_1+a_2$, i.e. $a_1+a_2 \in (\mathcal P_D+\mathcal U_D)_{n-1}$.  Hence we have 
\begin{equation*}
b=l(a_1+a_2)\in l(\mathcal P_D+\mathcal U_D)_{n-1}\subset (\mathcal P_D+\mathcal U_D)_n.
\end{equation*}
If the two terms, $b_1$ and $b_2$ do not have any common factor then with the same method as above we can rewrite the binomial $b$ by adding and subtracting a term of the determinant, $m$ of degree $n$, which has a common factor $m_1$ with $b_1$ and a common factor $m_2$ with $b_2$. Then we will have

\begin{equation*}
b_1+b_2= b_1+m+b_2-m=m_1(c_1+m')+m_2(c_2-m''),
\end{equation*}
where $b_1=m_1c_1$, $m=m_1m'=m_2m''$ and $b_2=m_2c_2$. Since $c_1+m'$ and $c_2-m''$ are of degree at most $n-1$, the induction assumption yields
\begin{equation*}
b_1+b_2=m_1(c_1+m')+m_2(c_2-m'')\in(\mathcal P_D+\mathcal U_D)_n.
\end{equation*}
This completes the induction step and hence the proof of the proposition.

\end{proof}
\begin{cor} \label{cor:generic-main-k} For a generic $n\times n$ matrix $A$ and each integer $k$, $1\leq k \leq n$, we have 
\begin{equation*}
(\mathcal P_D+\mathcal U_D)_k= \mathrm {Ann}(\det (A)) \cap S_k.
 \end{equation*}
 We also have  $(\mathcal U_D)_{n+1}=S_{n+1}$.
\end{cor}

\begin{proof}
Using equation \ref{eq:main-pre} we only need to show that 
\begin{equation*}
\mathrm{Ann}(\det (A)) \cap S_k \subset (\mathcal P_D+\mathcal U_D)_k.
\end{equation*}  

By Lemma \ref{lem:intro-gen} and Remark \ref{remark:introIK} we have
\begin{equation*}
(\mathrm{Ann}(\det (A)))_k=(\mathrm{Ann}(S_{n-k}\circ (\det (A))))_k=(\mathrm{Ann}(M_{k}(A)))_k
\end{equation*}

If we label the $k\times k$ minors of $A$ by $f_1,...,f_s$ we have

\begin{equation*}
(\mathrm{Ann}(M_{k}(A)))_k=\mathrm{Ann}(<f_1,...,f_s>)_k=(\bigcap_{i=1}^{i=s}(\mathrm{Ann}(f_i))_k
\end{equation*}

But for each $f_i$ if we denote the ring of variables of $f_i$ by $R^i$, then by Proposition \ref{prop:generic-n} we have
\begin{equation*}
(\mathcal P^i_D+\mathcal U^i_D)_{k}=\mathrm{Ann}(f_i)\cap S^i_k.
\end{equation*}

Hence

\begin{equation*}
\mathrm{Ann}(\det(A))\cap S_{k}\subset (\mathcal P_D+\mathcal U_D)_k.
\end{equation*}  

Finally, every monomial of degree larger than $n$ will be unacceptable. So we have  $(\mathcal U_D)_{n+1}=S_{n+1}$.

\end{proof}

\vspace{.2in}

\begin{thm}\label{thm:main-generic-determinant} Let A be a generic $n\times n$ matrix. Then the apolar ideal $\mathrm{Ann}(\det (A))\subset S $ is the ideal $(\mathcal P_D+\mathcal U_D)$, and is generated in degree two.
\par
\end{thm}

 \begin{proof} This follows directly from the Proposition \ref{prop:generic-n} and Corollary \ref{cor:generic-main-k}.
 \end{proof}
 \vspace{.2in}
 
\begin{thm}\label{thm:main-generic-permanent} Let A be a generic $n\times n$ matrix. Then the apolar ideal $\mathrm{Ann}(\mathrm{Per} (A))\subset S$ to $\mathrm{Per}(A)\in R$ is the ideal $(\mathcal M_D+\mathcal U_D)$, generated in degree two.

\end{thm}

\begin{proof}
The proof follows directly from the proof of the Proposition \ref{prop:generic-n} and Corollary \ref{cor:generic-main-k}, by interchanging the determinants and the permanents.
\end{proof}
\vspace{.2in}
\begin{cor} Let $A=(a_{ij})$ be an $m\times n$ matrix where $n\ge m$. Let $N$ denote the space generated by all $m\times m$ minors of $A$. Then $\mathrm{Ann}(N)$ is generated in degree two by all $2\times 2$ permanents of $A$ and the degree two unacceptable monomials.

\end{cor}

\begin{proof} 
Let $s={n\choose m}$, and $f_1,...,f_s$ denote the $m\times m$ minors of $A$. We have 

$$\mathrm{Ann}(N)=\mathrm{Ann}(<f_1,...,f_s>)=\bigcap_{i=1}^{i=s}(\mathrm{Ann}(f_i)).$$

Let $R^i$ denote the ring of variables of $f_i$. Hence by Theorem 2.12 we have $\mathrm{Ann}(f_i)\cap S^i $ is generated in degree 2. So we have $\mathrm{Ann}(N)$ is also generated in degree 2.

\end{proof}


\section{Application to the ranks of the determinant and the permanent}
\label{rank-generic}

\begin{Notation}

Let $F\in R=\sf k$$[a_{ij}]$ be a homogeneous form of degree $d$. A presentation 
\begin{equation}\label{eq:Waring decomposition}
F=l_1^d+...+l_s^d \text{ with }l_i\in R_1.
\end{equation}

is called a \emph{Waring decomposition} of length $s$ of the polynomial $F$. The minimal number $s$ that satisfies the Equation \ref{eq:Waring decomposition} is called the \emph{rank} of $F$.

The apolarity action of $S=\sf k$$[d_{ij}]$ on $R$, defines $S$ as a natural coordinate ring on the projective space $\textbf{P}(R_1)$ of 1-dimensional subspaces of $R_1$ and vice versa. A finite subscheme $\Gamma\subset \textbf{P}(R_1)$ is apolar to $F$ if the homogeneous ideal $I_\Gamma \subset S$ is contained in $\mathrm {Ann}(F)$ (\cite{IK},\cite{RS}).

\begin{remark}\label{rmrk:rank-definition-ranestad}((\cite{IK} Def. 5.66,\cite{RS})) Let $\Gamma=\{[l_1],...,[l_s]\}$ be a collection of $s$ points in $\textbf{P}(R_1)$. Then 
$$
F=c_1l_1^d+...+c_sl_s^d \text{ with }c_i\in\sf{k}.
$$
if and only if 
$$I_\Gamma \subset \mathrm {Ann}(F)\subset S$$

\end{remark}

\end{Notation}
\vspace{0.2in}
\begin{defi}\label{def:ranks} 

We have the following ranks (\cite{IK} Def. 5.66 , \cite{BR} and  \cite{RS}). Here $\Gamma$ is a punctual scheme (possibly not smooth), and the degree of $\Gamma$ is the number of points (counting multiplicities) in $\Gamma$.

a. the rank $r(F)$:
\begin{equation*}
r(F)=\min\{\deg \Gamma| \Gamma\subset \textbf{P}(R_1) \text{ smooth}, \dim \Gamma=0,I_\Gamma \subset \mathrm {Ann}(F)\}.
\end{equation*}

Note that when $\Gamma$ is smooth, it is the set of points in the Remark \ref{rmrk:rank-definition-ranestad} (\cite{IK}, page 135).

b. the smoothable rank  $sr(F)$:
\begin{equation*}
sr(F)=\min\{\deg \Gamma| \Gamma\subset \textbf{P}(R_1) \text{ smoothable}, \dim \Gamma=0,I_\Gamma \subset \mathrm {Ann}(F)\}.
\end{equation*}

Note that for the smoothable rank one considers the smoothable schemes, that are the schemes which are the limits of smooth schemes of $s$ simple points (\cite{IK}, Definition 5.66). 

c. the cactus rank (scheme length in \cite{IK}, Definition 5.1 page 135) $cr(F)$:
\begin{equation*}
cr(F)=\min\{\deg \Gamma| \Gamma\subset \textbf{P}(R_1), \dim \Gamma=0,I_\Gamma \subset \mathrm {Ann}(F)\}.
\end{equation*}


d. the differential rank (Sylvester's catalecticant or apolarity bound) is the maximal dimension of a homogeneous component of $S/\mathrm {Ann}(F)$:
\begin{equation*}
l_{diff}(F)= \max_{i\in \mathbb{N}_0}  \{ (H(S/\mathrm {Ann}(F)))_i\}.
\end{equation*}
\end{defi}

Note that we give a lower bound for the cactus rank of the determinant and permanent of the generic matrix. We do not have information on  the smoothable rank of the generic  determinant or permanent. It is still open to find a bound for the smoothable rank. The work of A. Bernardi and K. Ranestad \cite{BR} in the case of generic forms of a given degree and number of variables show that the cactus rank and smoothable rank can be very different.

\vspace{0.2in}
\begin{prop}\label{prop:ranks-inequality}(\cite{IK}, Proposition 6.7C)
The above ranks satisfy
\begin{equation*}
l_{diff}(F) \leq cr(F) \leq sr(F) \leq r(F).
\end{equation*}
\end{prop}

\begin{prop}\label{prop-mainRS} \textbf{(Ranestad-Schreyer)} If the ideal of $\mathrm {Ann}(F)$ is generated in degree d and $\Gamma\subset \textbf{P}(T_1)$ is a finite (punctual) apolar subscheme to $F$, then
\begin{equation*}
\deg \Gamma \geq \frac{1}{d} \deg (\mathrm {Ann}(F)),
\end{equation*}
where $\deg (\mathrm {Ann}(F)) =\dim (S/\mathrm {Ann} (F))$ is the length of the 0-dimensional scheme defined by $\mathrm {Ann}(F)$. 
\end{prop}
\vspace{.2in}

If in Proposition \ref{prop-mainRS} we take $F= \det (A)$ or $F=\mathrm{Per}(A)$, since we have found that for the determinant and the permanent of a matrix we have $d=2$; we can use the above proposition to find a lower bound for the above ranks of $F$.
\vspace{.2in}

\begin{thm}\label{generic-det-perm-RS-rank}
Let $F$ be the determinant or permanent of a generic $n\times n$ matrix $A$. We have

\begin{equation*}
{1\over{2}} {2n\choose n} \leq cr(F) \leq sr(F) \leq r(F).
\end{equation*}
\end{thm}

\begin{proof}
By Theorems \ref{thm:main-generic-determinant} and \ref{thm:main-generic-permanent}, Propositions \ref{prop-mainRS} and \ref{prop:ranks-inequality},and Equations \ref{eq:Hilbert-generic} and \ref{eq:sum-Hilbert-generic} we have for an apolar punctual scheme $\Gamma$,
\begin{equation*}
\deg \Gamma \geq \frac{1}{d} \deg (\mathrm {Ann}(F)= \frac{1}{2}\sum_{k=0}^{k=n} {{n\choose k}}^2= \frac{1}{2}{2n \choose n}.
\end{equation*}
\end{proof}
\vspace{.2in}

\begin{Notation} \cite{LT}
Let $\Phi \in S^d\mathbb C^n$ be a polynomial, we can polarize $\Phi$ and consider it as a multilinear form $\tilde{\Phi}$ where $\Phi(x)=\tilde{\Phi}(x,...,x)$ and consider the linear map $\Phi_{s,d-s}:S^s\mathbb C^{n*}\rightarrow S^{d-s}\mathbb C^{n}$, where $\Phi_{s,d-s}(x_1,...,x_s)(y_1,...,y_{d-s})=\tilde{\Phi}(x_1,...,x_s,y_1,...y_{d-s})$. Define
\begin{equation*}
Zeros(\Phi)=\{[x]\in \mathbb P \mathbb C^{n*}| \Phi(x)=0\} \subset  \mathbb P \mathbb C^{n*}.
\end{equation*}
Let $x_1,...,x_n$ be linear coordinates on  $\mathbb C^{n*}$ and define

$$
\Sigma_s(\Phi):=\{[x] \in Zeros(\Phi)| \frac{\partial^I \Phi}{\partial x^I}(x)=0,\forall  I,\text{ such that } |I|\leq s\}.
$$

\end{Notation}
\vspace{.2in}

In this notation $\Phi_{s,d-s}$ is the map from $S_s\to R_{n-s}$ taking  $h$ to $h\circ \Phi$, hence its rank is  $H(\mathfrak A_A)_s$.

In the following theorem we use the convention that $\dim \emptyset=-1$.
\vspace{.2in}

\begin{thm} \label{thm:LT-rank}\textbf{(Landsberg-Teitler)}(\cite{LT}) 
Let $\Phi \in S^d\mathbb C^{n}$, Let $1\leq s \leq d$. Then 

\begin{equation*}
rank(\Phi)\geq rank \Phi_{s,d-s}+ \dim \Sigma_s(\Phi)+1.
\end{equation*}
\end{thm}
\vspace{.2in}
\begin{rmk} (\textbf{Z. Teitler})
If we define $\Sigma_s(\Phi)$ to be a subset of affine rather than projective space, then the above theorem  does not need $+1$ at the end, and does not need the statement that the dimension of the empty set is $-1$.

\end{rmk}
\vspace{.2in}
Applying this theorem for the determinant yields
\vspace{.2in}
\begin{cor} \label{cor:-detLT-rank}\textbf{(Landsberg-Teitler)}(\cite{LT}) 

\begin{equation*}
{r({\det}_{n})} \geq {n \choose {\lfloor{n/2}\rfloor}}^2+n^2-{(\lfloor{n/2}\rfloor+1)}^2.
\end{equation*}
\end{cor}

\vspace{.2in}

\begin{prop} \label{prop:-BR-rank}\textbf{(Bernardi-Ranestad)}(\cite{BR}, Theorem 1) 
Let $F\in R^s$ be a homogeneous form of degree $d$, and let $l$ be any linear form in $S^s_1$. Let $F_l$ be a dehomogenization of $F$ with respect to $l$. Denote by $ \mathrm{Diff}(F)$ the subspace of $S^s$ generated by the partials of $F$ of all orders. Then
\begin{equation*}
cr(F)\leq \dim_{\sf k}  \mathrm{Diff}(F_l)
\end{equation*}

\end{prop}

We thank Pedro Marques for pointing out that it is easy to show that  the length of a polynomial is an upper bound for the length of any dehomogenization of that polynomial. So we have
\begin{equation}\label{eq:Pedro}
cr(F)\leq \dim_{\sf k} \mathrm{Diff}(F)= \deg (\mathrm {Ann}(F))
\end{equation}
\vspace{.2in}

\begin{prop}\label{prop:CCG-rank} 
 For the monomial $m=x_1^{b_1}...x_n^{b_n}$, where $1\leq b_1\leq...\leq b_n$ we have

(a) (\cite{CCG}) 
\begin{equation*}
r(x_1^{b_1}...x_n^{b_n})=\Pi_{i=2}^{i=n}(b_i+1)
\end{equation*}
(b) (\cite{RS}) 
\begin{equation*}
sr(x_1^{b_1}...x_n^{b_n})=cr(x_1^{b_1}...x_n^{b_n})=\Pi_{i=1}^{i=n-1}(b_i+1)
\end{equation*}
(c)(\cite{BBT2}) Let $d=b_1+\ldots+b_n$, and $m=l_1^d+\ldots+l_s^d$ with $r(m)=s$. Let $I\subset S$ be the homogeneous ideal of functions vanishing on $Q=\{[l_1],\ldots,[l_s]\}\subset \textbf{P}^{ n-1}$. Then $I$ is a complete intersection of degrees $b_2+1,\ldots,b_n+1$ generated by
$$y_2^{b_2+1}-\Phi_1y_1^{b_1+1},\ldots,y_n^{b_n+1}-\Phi_ny_1^{b_1+1},$$
for some homogeneous polynomials $\Phi_i\in S$ of degree $b_i-b_1$.

\end{prop}

\vspace{.2in}

\begin{ex} Let $n=2$, and \begin{center}  $ A= \left(%
\begin{array}{cc}
 
  a & b \\
  c & d\\
  \end{array}%
\right)$,

\end{center}

$\det(A)=ad-bc=(a+d)^2-(a-d)^2+(b-c)^2-(b+c)^2$ so $r(\det(A))=4$.

The corresponding Hilbert sequence for $n=2$ is $(1,4,1)$. We have $l_{diff}(\det(A))=4$. Using Theorem \ref{generic-det-perm-RS-rank} we have:
\begin{equation*}
cr(\det(A)) \geq \frac{1}{d} \deg (\mathrm {Ann}(\det(A)))=\frac{1}{2} (6)=3.
\end{equation*}
So the lower bound we obtain using Theorem \ref{generic-det-perm-RS-rank} is 3.
\newline\newline Using Corollary \ref{cor:-detLT-rank} (Landsberg-Teitler) we obtain:
\begin{equation*}
{r({det}_{2})} \geq {2 \choose {\lfloor{2/2}\rfloor}}^2+2^2-{(\lfloor{2/2}\rfloor+1)}^2=4+4-4=4.
\end{equation*}
On the other hand we have
\begin{equation*}
\det(A)=ad-bc=1/4((a+d)^2-(a-d)^2)-1/4((b+c)^2-(b-c)^2).
\end{equation*}
Hence

\begin{equation*}
r(\det(A))=cr(\det (A))=sr(\det (A))=l_{diff}(\det(A))=4.
\end{equation*}

\end{ex}

\vspace{.2in}
\begin{ex} Let $n=3$, and \begin{center}  $ A= \left(%
\begin{array}{ccc}
 
  a & b & e \\
  c & d  & f  \\
  g & h & i \\
  \end{array}%
\right)$,

\end{center}

$\det(A)=g(bf-de)-h(af-ce)+i(ad-bc)$.

Using Macaulay2 for the calculations we obtain the Hilbert sequence $(1,9,9,1)$, and by Theorem \ref{generic-det-perm-RS-rank} we have:
\begin{equation*}
cr(\det(A)) \geq \frac{1}{d} \deg (\mathrm {Ann}(\det(A)))=\frac{1}{2} (20)=10.
\end{equation*}
So the lower bound we find using the Theorem \ref{generic-det-perm-RS-rank} is 10, which is greater than the $l_{diff}(\det(A))=9$, so it is a better lower bound for the cactus and smoothable ranks introduced above.
\newline\newline Using Corollary \ref{cor:-detLT-rank} we have:

\begin{equation*}
{r({det}_{3})} \geq {3 \choose {\lfloor{3/2}\rfloor}}^2+3^2-{(\lfloor{3/2}\rfloor+1)}^2=9+9-4=14.
\end{equation*}


On the other hand, for every $x$, $y$ and $z$, it is easy to see that $r(xyz) \leq4$:
\begin{equation*}
xyz=1/24({(x+y+z)}^3+{(x-y-z)}^3-{(x-y+z)}^3-(x+y-z)^3).
\end{equation*}
Hence $14 \leq r(\det(A))\leq 24$.

If $a=1$ in $\det(A)$, we have that the punctual scheme $\mathrm{Ann}(\det A_{a=1})$ of degree $18$ with Hilbert function $(1,8,8,1)$. So by Proposition  \ref{prop:-BR-rank} we have:

$$cr(\det(A))\leq 18.$$

\end{ex}
\vspace{.2in}

\begin{ex} Let $n=4$, and \begin{center}  $ A= \left(%
\begin{array}{cccc}
 
  a & b & e & j\\
  c & d  & f  & k\\
  g & h & i  & l \\
  m & n & o & p\\
  \end{array}%
\right)$,

\end{center}

 
Using Macaulay2 for the calculations we obtain the Hilbert sequence $(1,16,36,16,1)$. By Theorem \ref{generic-det-perm-RS-rank}, 
\begin{equation*}
cr(\det(A)) \geq \frac{1}{d} \deg (\mathrm {Ann}(\det(A)))=\frac{1}{2} (70)=35.
\end{equation*}
which is less than the $l_{diff}(\det(A))=36$. So in this case $l_{diff}$ is a better lower bound for the cactus rank.

Using Corollary \ref{cor:-detLT-rank} (Landsberg-Teitler) we have:

\begin{equation*}
{r({\det}_{4})} \geq {4 \choose {\lfloor{4/2}\rfloor}}^2+4^2-{(\lfloor{4/2}\rfloor+1)}^2=36+16-9=43.
\end{equation*}

So the lower bound found by Corollary \ref{cor:-detLT-rank} (Landsberg-Teitler) is a better lower bound for the rank in this case.

Now using Proposition \ref{prop:CCG-rank} we have 

\begin{equation*}
{r({\det}_{4})} \leq (4!) (2^3)=192
\end{equation*}

\end{ex}
\vspace{.2in}


\begin{ex} Let $n=5$, and \begin{center}  $ A= \left(%
\begin{array}{ccccc}
 
  a & b & e & j & q\\
  c & d  & f  & k & r\\
  g & h & i  & l  & s\\
  m & n & o & p &t\\
  u  & v & w & x & y\\
  \end{array}%
\right)$,

\end{center}

 
Using Macaulay2 for the calculations we obtain the Hilbert sequence $(1,25,100,100,25,1)$. By Theorem  \ref{generic-det-perm-RS-rank}
\begin{equation*}
cr(\det(A))  \geq \frac{1}{d} \deg (\mathrm {Ann}(\det(A)))=\frac{1}{2} (252)=126,
\end{equation*}
which is greater than the $l_{diff}(\det(A))=100$. So it is a better lower bound for cactus rank than $l_{diff}$.

Using Corollary \ref{cor:-detLT-rank} (Landsberg-Teitler) we have:

\begin{equation*}
{r({\det}_{5})} \geq {5 \choose {\lfloor{5/2}\rfloor}}^2+5^2-{(\lfloor{5/2}\rfloor+1)}^2=116.
\end{equation*}

So for the first time at $n=5$ Theorem \ref{generic-det-perm-RS-rank} gives us a better lower bound for the rank than Corollary \ref{cor:-detLT-rank} (Landberg-Teitler).

Now using Proposition \ref{prop:CCG-rank} we have 

\begin{equation*}
{r({\det}_{5})} \leq (5!) (2^4)=1920
\end{equation*}

\end{ex}
\vspace{.2in}

\begin{ex} Let $n=6$. Using Macaulay2 for the calculations we obtain the Hilbert sequence  $$H(S/\mathrm{Ann}(\det A))=(1,36,225,400,225,36,1).$$ Now using Theorem  \ref{generic-det-perm-RS-rank} we have:
\begin{equation*}
cr(\det(A)) \geq \frac{1}{d} \deg (\mathrm {Ann}(\det(A)))=\frac{1}{2} (924)=462.
\end{equation*}
So the lower bound we can find using Theorem  \ref{generic-det-perm-RS-rank} is 462, which is greater than the $l_{diff}(\det(A))=400$, and therefore is a better lower bound for cactus rank than $l_{diff}$.

Using Corollary \ref{cor:-detLT-rank} (Landsberg-Teitler) we have:

\begin{equation*}
{r({\det}_{6})} \geq {6 \choose {\lfloor{6/2}\rfloor}}^2+6^2-{(\lfloor{6/2}\rfloor+1)}^2=420.
\end{equation*}

So again at $n=6$ Theorem \ref{generic-det-perm-RS-rank} give us a better lower bound than Corollary \ref{cor:-detLT-rank} (Landberg-Teitler).

Now using Proposition \ref{prop:CCG-rank} we have 

\begin{equation*}
{r({\det}_{6})} \leq (6!) (2^5)=23040
\end{equation*}

\end{ex}
\vspace{.2in}

\begin{remark}\label{remark-stirling-generic}
(a)Using Stirling's formula, $n!\sim\sqrt{2\pi n}\left({\frac{n}{e}}\right)^n$, we can approximate $2n\choose n$ for large $n$ by $4^n/\sqrt{n\pi}$. Hence for large $n$ Theorem \ref{generic-det-perm-RS-rank} gives us a lower bound asymptotic to $4^n/2\sqrt{n\pi}\leq cr(\det(A))$, and the Landsberg-Teitler formula gives us the lower bound $2\cdot~4^n/(n\pi)\leq r(\det(A))$. The Landsberg-Teitler lower bound for $r(\det(A))$ is also asymptotic to $l_{diff}(\det (A))={n \choose {\lfloor{n/2}\rfloor}}^2$, which is a lower bound for $ cr(\det (A)).$ These are also lower bounds for the corresponding ranks of the permanent of a generic $n\times n$ matrix.\par

(b) Using Proposition \ref{prop:CCG-rank} the upper bound for the rank of the determinant and permanent of a generic $n\times n$ matrix is given by $(n!)2^{n-1}$. This can be approximated for large $n$, using Stirling's formula, by $\sqrt{2\pi n}{(\frac{n}{e})}^n(2^{n-1})$.

(c) By Equation \ref{eq:Pedro} an upper bound for the cactus rank of both the determinant and permanent of a generic $n\times n$ matrix is $2n\choose n$, which is asymptotic to $4^n/\sqrt {n \pi}$.

\end{remark}
In the following table we give lower bounds for the ranks of the determinant and permanent of an $n\times n$ generic matrix.
\begin{table}[h]
\begin{center}
\caption{The determinant of the generic matrix}\label{table:generic}
\begin{tabular}{l*{7}{c}r}
\hline
$n$              & 2 & 3 & 4 & 5 & 6  & $n\gg 0$\\
\hline
lower bound for $cr(\det(A))$ by Theorem \ref{generic-det-perm-RS-rank} & 3 & 10 & 35 & 126 & 462 & $4^n/2\sqrt{n\pi}$ \\
\hline
lower bound for $r(\det(A))$ by Corollary \ref{cor:-detLT-rank}          & 4 & 14 & 43 &  116 & 420 & $4^n/2n\pi$\\
\hline
$l_{diff}(\det(A))$& 4 & 9 & 36 & 100 & 400 & ${n \choose {\lfloor{n/2}\rfloor}}^2$\\
\hline

\end{tabular}
\end{center}
\end{table}


\section{Annihilator of the Pfaffian and Hafnian}
\label{Annihilator of the Pfaffian and Hafnian}

In this section we discuss the annihilator ideals of the Pfaffians and of the Hafnians. We show that the annihilator ideal of the Pfaffian of a generic skew symmetric $2n\times 2n$ matrix  and the annihilator ideal of the Hafnian of generic symmetric $2n\times 2n$ matrix are both generated in degree 2. 

In the following discussion we let $X_m^{sk}=(x_{ij})$ with $x_{ij}=-x_{ji}$ be an  $m\times m$ skew symmetric matrix of indeterminates in the polynomial ring $R^{sk}=\sf k$$[x_{ij}]$, Let $Y_m^{sk}=~(y_{ij})$ with $y_{ij}=-y_{ji}$ be an  $m\times m$ skew symmetric matrix of indeterminates in the ring of differential operators $S^{sk}=\sf k$$[y_{ij}]$. We denote the Pfaffian of the matrix  $X_m^{sk}$ by $Pf(X_m^{sk})$. It is well known that for any odd number $m$ we have $\det(X_m^{sk})=0$. It is also well known that the square of the Pfaffian is equal to the determinant of a skew symmetric matrix. So in the following we are going to consider the annihilator of the Pfaffian of generic $m\times m$ skew~symmetric matrices, where $m=2n$ is an even number. Recall that
\vspace{0.2in}
\begin{Notation}
Let $F_{2n}\subset S_{2n}$ be the set of all permutations $\sigma$ satisfying the following conditions:

(1) $\sigma(1)<\sigma(3)<...<\sigma(2n-1)$

(2) $\sigma(2i-1)<\sigma(2i)$ for all $1\leq i \leq n$

\begin{itemize}

\item For a $2n\times 2n$ generic skew symmetric matrix $X^{sk}$, we denote by $Pf(X^{sk})$ the Pfaffian of  $X^{sk}$ defined by

\begin{equation}\label{eq:definition of Pfaffian-4}
Pf(X^s)=\sum_{\sigma\in F_{2n}} sgn({\sigma}) x_{\sigma(1)\sigma(2)}x_{\sigma(3)\sigma(4)}...x_{\sigma(2n-1)\sigma(2n)}
\end{equation}

\item
(\cite{IKO})  We denote by $Hf(X^s)$ the Hafnian of a generic symmetric $2n\times 2n$ matrix $X^s$ defined by

\begin{equation}\label{eq:definition of Hafnian-4}
Hf(X^s)=\sum_{\sigma\in F_{2n}} x_{\sigma(1)\sigma(2)}x_{\sigma(3)\sigma(4)}...x_{\sigma(2n-1)\sigma(2n)}
\end{equation}

\end{itemize}

\end{Notation}


Let $J_{2n}=\mathrm{Ann}(Pf(X_{2n}^{sk}))$. We first give some examples and then some partial results concerning $\mathrm{Ann}(Pf(X_{2n}^{sk}))$. Using Macaulay2 for calculations  we have the following results:

(a) Let $X_2$ be a generic skew symmetric $2\times 2$ matrix, then we have $H(S^{sk}/J_2)=(1,1)$. And the maximum degree of the generators of the annihilator ideal $J_2$ is 2. So using the Ranestad-Schreyer Proposition we have: 
\begin{equation*}
cr(Pf(X_2^{sk})\ge \frac{1}{d}\deg(\mathrm{Ann}(Pf(X_2^{sk})))=\frac{1}{2}(2)=1,
\end{equation*}
which is the same as the  differential length in this case. Evidently, in this case $r(Pf(X_2^{sk})=~1$, so we have
\begin{equation*}
r(Pf(X_2^{sk})=cr(Pf(X_2^{sk})=sr(Pf(X_2^{sk})=l_{diff}(Pf(X_2^{sk})=1.
\end{equation*}


(b) Let $X_4$ be a generic skew symmetric $4\times 4$ matrix. Using Macaulay2 for calculations we have $H(S^{sk}/J_4)=(1,6,1)$, and the maximum degree of the generators of the annihilator ideal $J_4$ is 2. Using the Ranestad-Schreyer Proposition we have: 
\begin{equation*}
cr(Pf(X_4^{sk})\ge \frac{1}{d}\deg(\mathrm{Ann}(Pf(X_4^{sk})))=\frac{1}{2}(8)=4,
\end{equation*}

which is less than $l_{diff}=6$.


(c) Let $X_6$ be a generic skew symmetric $6\times 6$ matrix. Using Macaulay2 for calculations we have  $H(S^{sk}/J_6)=(1,15,15,1)$, and the maximum degree of the generators of the annihilator ideal $J_6$ is 2. Using the Ranestad-Schreyer Proposition we have: 
\begin{equation*}
cr(Pf(X_6^{sk})\ge \frac{1}{d}\deg(\mathrm{Ann}(Pf(X_6^{sk})))=\frac{1}{2}(32)=16,
\end{equation*}

which is larger than $l_{diff}=15$.


(d) Let $X_8$ be a generic skew symmetric $8\times 8$ matrix. Using Macaulay2 for calculations we have $H(S^{sk}/J_8)=(1,28,70,28,1)$, and the maximum degree of the generators of the annihilator ideal $J_8$ is 2. From the Ranestad-Schreyer Proposition we have: 
\begin{equation*}
cr(Pf(X_8^{sk})\ge \frac{1}{d}\deg(\mathrm{Ann}(Pf(X_8^{sk})))=\frac{1}{2}(128)=64,
\end{equation*}

which is less than $l_{diff}=70$.

(e) Let $X_{10}$ be a generic skew symmetric $10\times 10$ matrix. Using Macaulay2 for calculations we have  $$H(S^{sk}/J_{10})=(1,45,210,210,45,1).$$ The maximum degree of the generators of the annihilator ideal $J_{10}$ is 2. From the Ranestad-Schreyer Proposition we have: 

\begin{equation*}
cr(Pf(X_{10}^{sk})\ge \frac{1}{d}\deg(\mathrm{Ann}(Pf(X_{10}^{sk})))=\frac{1}{2}(512)=256,
\end{equation*}

which is larger than $l_{diff}=210$.







\vspace{.2in}
\begin{remark}\label{remark: Hilbert-Pfaffian}
The Hilbert sequence for the apolar algebra of the Pfaffian of a generic $2n\times 2n$ matrix is given by $2n\choose2t$, and we have $\sum_{t=0}^{t=n}{{2n}\choose {2t}}=2^{2n-1}$.
\end{remark}
\vspace{.2in}
\begin{defi}\label{def: 2t-Pfaffian}
A $2t$-Pfaffian minor of a skew symmetric matrix $X$ is a Pfaffian of a submatrix of $X$ consisting of rows and columns indexed by $i_1,i_2,...,i_{2t}$ for some $i_1<i_2<...<i_{2t}$.
\end{defi}

The number of $2t$-Pfaffian minors of a $2n\times 2n$ skew symmetric matrix is clearly $2n\choose 2t$. We denote by $\{P_{2t}(X^{sk})\}$ the set of the $2t$-Pfaffians of $X^{sk}$. Furthermore, we denote by $P_{2t}(X^{sk})$ the vector space generated by $\{P_{2t}(X^{sk})\}$ in $R^{sk}_t$ and we denote by $(P_{2t}(X^{sk}))$ the ideal generated by $\{P_{2t}(X^{sk})\}$ in $R^{sk}$. Let $\tau$ be the lexicographic term order on $R^{sk}=\sf k$$[x_{ij}]$ induced by the following order on the indeterminates:

\begin{equation*}
x_{1,2n}\ge x_{1,2n-1}\ge...\ge x_{1,2}\ge x_{2,2n}\ge x_{2,2n-1} \ge...\ge x_{2n-1,2n}.
\end{equation*}





\vspace{.2in}
\begin{thm}\label{thm:Herzog-Trung basis Pfaffian}(Herzog-Trung \cite{HT}, Theorem 4.1)
The set $\{P_{2t}(X)\}$ of the $2t$-Pfaffians of the matrix $X^{sk}$ is a Gr\"{o}bner basis of the ideal $(P_{2t}(X))$ with respect to $\tau$.

\end{thm}
\vspace{.2in}
\begin{cor}\label{cor:dimension-degree-Pfaffian-Hilbert}
The dimension of the space of $2t \times 2t$ Pfaffians of a $2n\times 2n$ generic skew symmetric matrix $X^{sk}$ is $2n\choose 2t$. So we have
\begin{equation*}
\dim(S^{sk}/\mathrm{Ann}(Pf(X^{sk})))=2^{2n-1}.
\end{equation*}
\end{cor}

\begin{proof}
The proof follows directly from the Theorem \ref{thm:Herzog-Trung basis Pfaffian} and the combinatorial identity:

$$\sum_{t=0}^{t=n}{{2n}\choose {2t}}=2^{2n-1}.$$

This identity is easy to show; e.g., it follows immediately by evaluating at $x=1$ and $x=-1$ the binomial expansion of $(x+1)^{2n}$.
\end{proof}

\vspace{.2in}


The examples strongly suggest that the apolar ideal of the Pfaffian is generated in degree~$2$. In the remaining part of this section we prove that this is always the case.

\vspace{.2in}
\begin{defi}\label{def:PfaffianW} Let $W$ be the vector subspace of $S^{sk}$ spanned by degree 2 elements of type (a), (b) and (c) defined as follows

(a) square of each element of $Y^{sk}$. The number of these monomials is $2n^2-n$.

(b) product of each element of $Y^{sk}$ with another element in the same row or column of the matrix $Y^{sk}$. The number these monomials is $(2n^2-n)(2n-2)$.

(c) Given any $4\times 4$ submatrix of $X^{sk}$ of the rows and columns $i_1,i_2,i_3$ and $i_4$,

\begin{center}  $ Q=\left(%
\begin{array}{cccc}
  0 & x_{i_1i_2} & x_{i_1i_3} & x_{i_1i_4}  \\
  -x_{i_1i_2} & 0  & x_{i_2i_3} & x_{i_2i_4}\\
  -x_{i_1i_3} & -x_{i_2i_3} & 0 & x_{i_3i_4}\\
  -x_{i_1i_4} & -x_{i_2i_4} & -x_{i_3i_4} & 0\\
\end{array}%
\right),$

\end{center}

we have $Pf(Q)=x_{i_1i_2}x_{i_3i_4}-x_{i_1i_3}x_{i_2i_4}+x_{i_1i_4}x_{i_2i_3}$. Corresponding to $Pf(Q)$ we have 3 binomials which annihilate $Pf(Q)$ hence annihilate $Pf(X^{sk})$. These binomials are $y_{i_1i_2}y_{i_3i_4}+y_{i_1i_3}y_{i_2i_4}$, $y_{i_1i_2}y_{i_3i_4}-y_{i_1i_4}y_{i_2i_3}$ and $y_{i_1i_3}y_{i_2i_4}+y_{i_1i_4}y_{i_2i_3}$. However these three binomials are not linearly independent, and we can write one of them as the sum of the other 2 binomials. So corresponding to each $4\times 4$ Pfaffian we have 2 linearly independent binomials in the annihilator ideal, and using Theorem \ref{thm:Herzog-Trung basis Pfaffian}, the number of these binomials is $2\cdot{{2n}\choose 4}$.
\end{defi}

\vspace{.2in}
\begin{rmk}For  a $2n\times 2n$ skew symmetric matrix $X^{sk}$, we have $W\subset \mathrm{Ann}(Pf(X^{sk}))$.

\end{rmk}
\vspace{.2in}
\begin{lem}\label{lem:Pfaffian-gen2W} For the generic skew symmetric $2n\times 2n$ matrix $X^{sk}$, we have

\begin{equation*}
W=\mathrm{Ann}(P_4(X^{sk}))\cap S^{sk}_2.
\end{equation*}

\end{lem}

\begin{proof}
The monomials of type (a) and (b) correspond to unacceptable monomials discussed earlier and are linearly independent from any binomial in (c). The binomials in (c) are linearly independent by Theorem \ref{thm:Herzog-Trung basis Pfaffian}. Hence we have 

\begin{equation}\label{eq:Pf-W-proof-1}
\dim_{\sf k}(W)=2{{2n}\choose 4}+(2n^2-n)(2n-2)+2n^2-n={{2n^2-n+1}\choose 2}-{2n\choose4}.
\end{equation}

According to Remark \ref{remark:introIK} we have

\begin{equation*}
\dim_{\sf k}(\mathrm{Ann}(P_4(X^{sk})))\cap S^{sk}_2=\dim_{\sf k}S^{sk}_2-\dim_{\sf k}(P_4(X^{sk})).
\end{equation*}

So we have 

\begin{equation}\label{eq:Pf-W-proof-2}
\dim_{\sf k}(\mathrm{Ann}(P_4(X^{sk})))\cap S^{sk}_2={{2n^2-n+1}\choose 2}-{2n\choose4}.
\end{equation}

Using Equations \ref{eq:Pf-W-proof-1} and \ref{eq:Pf-W-proof-2} we obtain

\begin{equation}\label{eq:Pf-W-proof-3}
\dim_{\sf k}(W)=\dim_{\sf k}(\mathrm{Ann}(P_4(X^{sk})))\cap S^{sk}_2.
\end{equation}

On the other hand, evidently we have 

\begin{equation}\label{eq:Pf-W-proof-4}
W\subset \mathrm{Ann}(P_4(X^{sk})))\cap S^{sk}_2.
\end{equation}

Using Equations \ref{eq:Pf-W-proof-3} and \ref{eq:Pf-W-proof-4} we have 

\begin{equation*}
W=\mathrm{Ann}(P_4(X^{sk}))\cap S^{sk}_2.
\end{equation*}
\end{proof}
\begin{lem} \label{lem:Pfaffian-Sn-2circlePfaf}Let $X^{sk}$ be a $2n\times 2n$ skew symmetric matrix ($n\ge 2$). We have,
\begin{equation*}
S_{n-2}\circ Pf(X^{sk})=P_4(X^{sk})\subset R^{sk}_2.
\end{equation*}
\end{lem}
\begin{proof}

First we show 
\begin{equation}\label{eq:Pfaffian-lemproofSn-2-circle}
S_{n-2}\circ Pf(X^{sk})\supset P_4(X^{sk}).
\end{equation}

We use induction on the size of the matrix. 

The first step is $2n=6$. We denote by $f=[i_1,i_2,i_3,i_4]\in P_4(X^{sk})$ the Pfaffian of the sub matrix with the rows and columns $i_1,i_2,i_3$ and $i_4$. We have ${6\choose 4}=15$ choices for $f$. For any of these choices we get the Pfaffian of a $2\times 2$ sub matrix of the form 
\begin{center}  $ \left(%
\begin{array}{cc}
  0 & x \\
  -x & 0  \\
 \end{array}%
\right),$

\end{center}
as the coefficient of $f$ in the Pfaffian of the matrix $X^{sk}$. So if we differentiate the $6\times 6$ Pfaffian with respect to that variable $x$, we get the $4\times 4$ Pfaffian $f=[i_1,i_2,i_3,i_4]$.

Assume that Equation \ref{eq:Pfaffian-lemproofSn-2-circle} holds for the generic skew symmetric $(2n-2)\times(2n-2)$ matrix. We want to show it holds for the $2n\times 2n$ generic skew symmetric matrix. The Pfaffian of the skew symmetric $2n\times 2n$ matrix $X^{sk}$ can be computed recursively as 

\begin{equation}\label{eq:Pfaffian-lemproofSn-2-circle-eq2}
Pf(X^{sk})=\sum_{i=2}^{i=2n}(-1)^i x^{sk}_{1i} Pf(X^{sk}_{\hat{1}\hat{i}}),
\end{equation}

where $X^{sk}_{\hat{1}\hat{i}}$ denotes the matrix $X^{sk}$ with both the first and the $i$-th rows and columns removed. So $X^{sk}_{\hat{1}\hat{i}}$ is a $(2n-2)\times (2n-2)$ matrix and Equation \ref{eq:Pfaffian-lemproofSn-2-circle} holds for it. So for each choice of $[i_1,i_2,i_3,i_4]$ of the matrix $X^{sk}_{\hat{1}\hat{i}}$ we can find $n-3$ variables of  $X^{sk}_{\hat{1}\hat{i}}$ such that differentiating $Pf(X^{sk}_{\hat{1}\hat{i}})$ with respect to those variables gives us $[i_1,i_2,i_3,i_4]$. If we call those variable $a_1$,...,$a_{n-3}$, then using Equation \ref{eq:Pfaffian-lemproofSn-2-circle-eq2} if we add $x^{sk}_{1i}$ to our set of $n-3$ variables we will have a set of $n-2$ variables such that differentiating $Pf(X^{sk})$ with respect to those $n-2$ variables we will get $[i_1,i_2,i_3,i_4]$. Since we could write the recursive formula for the Pfaffian with respect to any other row or column, the result follows. 

For the opposite inclusion to Equation \ref{eq:Pfaffian-lemproofSn-2-circle} we have
\begin{equation*}
W\subset (\mathrm{Ann}(Pf(X^{sk})))_2\subset (\mathrm{Ann}(P_4(X^{sk})))_2.
\end{equation*}

But we have shown in Lemma \ref{lem:Pfaffian-gen2W} that
\begin{equation*}
W=(\mathrm{Ann}(P_4(X^{sk}))_2.
\end{equation*}

So we have 

\begin{equation*}
(\mathrm{Ann}(Pf(X^{sk})))_2=(\mathrm{Ann}(P_4(X^{sk})))_2.
\end{equation*}

By Remark \ref{remark:introIK} we have
\begin{equation*}
(\mathrm{Ann}(Pf(X^{sk})))_2=\mathrm{Ann}(S_{n-2}\circ (Pf(X^{sk}))).
\end{equation*}

Hence we have
\begin{equation*}
S_{n-2}\circ Pf(X^{sk})=P_4(X^{sk}).
\end{equation*}
\end{proof}

Recall that we denote by $P_{2k}(X^{sk})$ the vector subspace of $R^{sk}$ spanned by the $2k-$Pfaffian minors of $X^{sk}$ [Definition \ref{def: 2t-Pfaffian}]. 
\vspace{.2in}
\begin{lem}\label{lem:SkcirclePfaff} For $1\leq k\leq n-1$ we have
\begin{equation}\label{eq:SkcirclePfaff}
S_k\circ (Pf(X^{sk}))=P_{2n-2k}(X^{sk}).
\end{equation}
\end{lem}

\begin{proof}
First we want to show
\begin{equation*}
S_k\circ (Pf(X^{sk}))\subset P_{2n-2k}(X^{sk}).
\end{equation*}

We use induction on $k$. For $k=1$, we need to prove
\begin{equation*}
S_1\circ(Pf(X^{sk}))\subset P_{2n-2}(X^{sk}).
\end{equation*}
so we need to show for any monomial $y_{ij}\in S_1$ we have 
\begin{equation*}
y_{ij}\circ(Pf(X^{sk}))\subset P_{2n-2}(X^{sk}).
\end{equation*}

It is enough to show the above inclusion holds for $y_{12}$. Using equation \ref{eq:Pfaffian-lemproofSn-2-circle-eq2} we have
\begin{equation*}
y_{12}\circ(Pf(X^{sk}))=y_{12}\circ \sum_{i=2}^{i=2n}(-1)^i x^{sk}_{1i} Pf(X^{sk}_{\hat{1}\hat{i}})=Pf(X^{sk}_{\hat{1}\hat{2}})+\sum_{i=3}^{i=2n}(-1)^i x^{sk}_{1i} Pf(X^{sk}_{\hat{1}\hat{i}})\in P_{2n-2}(X^{sk}).
\end{equation*}

So indeed
\begin{equation*}
S_1\circ(Pf(X^{sk}))\subset P_{2n-2}(X^{sk}).
\end{equation*}

Next assume $S_k\circ (Pf(X^{sk}))\subset P_{2n-2k}(X^{sk})$. We want to show

\begin{equation*}
S_{k+1}\circ (Pf(X^{sk}))\subset P_{2n-2k-2}(X^{sk}).
\end{equation*}

We have 
\begin{equation*}
S_{k+1}\circ (Pf(X^{sk}))=S_{1}\circ (S_k \circ (Pf(X^{sk}))\subset S_1\circ (P_{2n-2k}(X^{sk})\subset P_{2n-2k-2}(X^{sk}).
\end{equation*}

For the other inclusion, we again use induction on $k$. First we show the inclusion holds for $k=1$. Let $\eta\in P_{2n-2}(X^{sk})$ be a $(2n-2)\times (2n-2)$ Pfaffian minor of $X^{sk}$. Corresponding to $\eta$ there exists a $2\times 2$ matrix of the form 
\begin{center}  $ \left(%
\begin{array}{cc}
  0 & x \\
  -x & 0  \\
 \end{array}%
\right),$
\end{center}
where $x$ is not in the $2n-2$ rows and columns of $\eta$. If we differentiate the Pfaffian of $X^{sk}$ with respect to $x$ we will get $\eta$. So we have $\eta \in S_1\circ (Pf(X^{sk}))$.

Next assume $P_{2n-2k}(X^{sk})\subset S_k\circ (Pf(X^{sk}))$, we have 

\begin{equation*}
P_{2n-2k-2}(X^{sk})\subset S_1\circ (P_{2n-2k}(X^{sk}))\subset S_1\circ (S_k\circ (Pf(X^{sk})))=S_{k+1}\circ (Pf(X^{sk})).
\end{equation*}

Thus by induction the equality holds.

\end{proof}

Recall that $(W)$ is the ideal of $S^{sk}$ spanned by degree 2 elements of type (a), (b) and (c) as in Definition \ref{def:PfaffianW}. 
\vspace{.2in}
\begin{prop}\label{prop:Pfaffian-main-prop-n}
For the $2n\times 2n$ generic skew symmetric matrix $X^{sk}$ we have
\begin{equation}
(W)_n=\mathrm{Ann}(Pf(X^{sk}))\cap S^{sk}_n
\end{equation}
\end{prop}
\begin{proof}
Let $2\leq k\leq n$. By Remark \ref{remark:introIK} and Lemma \ref{lem:SkcirclePfaff} we have

(1) $W \circ Pf(X^{sk})=0  \Longleftrightarrow  W \circ S^{sk}_{n-2}Pf(X^{sk})=0 \Longleftrightarrow W \circ P_{4}(X^{sk})=0$.

(2) ($\mathrm {Ann}(Pf(X^{sk}))) \cap S_2= W \Rightarrow S_{k-2} W \circ (S_{n-k} \circ Pf(X^{sk}))=0$.
\newline$ \Rightarrow S_{k-2}(W) \circ P_{2k}(X^{sk})=0$.
\newline$ \Rightarrow (W)_k \circ P_{2k}(X^{sk})=0$. 

Therefore for all integers $k$, $2\leq k\leq n$, we have

\begin{equation}\label{eq:Pfaffian-main-easy-k-inclusion}
(W)_k \subset \mathrm {Ann}(P_{2k}(X^{sk}))  \cap S^{sk}_k.
 \end{equation}
 
We need to show
 
\begin{equation}\label{eq:Pfaffian-main-difficult-n-inclusion}
(W)_n\supset \mathrm{Ann}(Pf(X^{sk}))\cap S^{sk}_n.
\end{equation}

We use induction on $n$. For $n=1,2$, we have the $2\times 2$ and $4\times 4$ skew symmetric matrices and the equality is easy to see. Now we want to show that the proposition holds for $n=3$. 

We use the Remark \ref{remark:main-ann}. Let $\eta$ be a binomial in $\mathrm{Ann}(Pf(X^{sk}))\cap S^{sk}_3$. Without loss of generality we can write 
\begin{equation*}
\eta=y_{12}y_{34}y_{56}-y_{\sigma(1)\sigma(2)}y_{\sigma(3)\sigma(4)}y_{\sigma(5)\sigma(6)}.
\end{equation*}

Where $\sigma \in S_6$, $sgn(\sigma)=1$ and we have $\sigma(1)<\sigma(3)<\sigma(5)$ and $\sigma(1)<\sigma(2)$, $\sigma(3)<\sigma(4)$ and $\sigma(5)<\sigma(6)$.

If the two terms of the binomial $\eta$ have a common factor then without loss of generality we can assume that the common factor is $y_{12}$ so we can write $\eta$ as
\begin{equation*}
\eta=y_{12}(y_{34}y_{56}-y_{\sigma(3)\sigma(4)}y_{\sigma(5)\sigma(6)})
\end{equation*}

But by the definition of $(W)_3$ the monomial $y_{34}y_{56}-y_{\sigma(3)\sigma(4)}y_{\sigma(5)\sigma(6)}$ is included in $W$ since it is of the form (c). So we have $\eta\in (W)_3$.

On the other hand, assume that the two terms of $\eta$, i.e. $y_{12}y_{34}y_{56}$ and $y_{\sigma(1)\sigma(2)}y_{\sigma(3)\sigma(4)}y_{\sigma(5)\sigma(6)}$ do not have any common factor. We can add and subtract another term of the Pfaffian $\tau=y_{\beta(1)\beta(2)}y_{\beta(3)\beta(4)}y_{\beta(5)\beta(6)}$ such that $\beta$ is a permutation in $S_6$ and we have $\beta(1)<\beta(3)<\beta(5)$ and $\beta(1)<\beta(2)$, $\beta(3)<\beta(4)$ and $\beta(5)<\beta(6)$. and $\tau$ has one common factor with $y_{12}y_{34}y_{56}$ and one common factor with $y_{\sigma(1)\sigma(2)}y_{\sigma(3)\sigma(4)}y_{\sigma(5)\sigma(6)}$. Without loss of generality we can take $\beta(5)=5, \beta(6)=6$ and $\beta(1)=\sigma(1), \beta(2)=\sigma(2)$. So we have 

\begin{equation*}
\eta -\tau+\tau=\eta-y_{\sigma(1)\sigma(2)}y_{\beta(3)\beta(4)}y_{5,6}+y_{\sigma(1)\sigma(2)}y_{\beta(3)\beta(4)}y_{5,6}.
\end{equation*}

Hence we have

\begin{equation*}
\eta=y_{5,6}(y_{12}y_{34}-y_{\sigma(1)\sigma(2)}y_{\beta(3)\beta(4)})+y_{\sigma(1)\sigma(2)}(y_{\beta(3)\beta(4)}y_{5,6}-y_{\sigma(3)\sigma(4)}y_{\sigma(5)\sigma(6)}).
\end{equation*}

But by the definition of $W$ we know that $y_{12}y_{34}-y_{\sigma(1)\sigma(2)}y_{\beta(3)\beta(4)}$  and $y_{\beta(3)\beta(4)}y_{5,6}-y_{\sigma(3)\sigma(4)}y_{\sigma(5)\sigma(6)}$ are both elements of $W$ of type (c). So we have $\eta\in (W)_3$.

When $n$ is larger than 3 then by the induction assumption we can assume that the proposition holds for all integers $k\leq n-1$. Again we use the Remark \ref{remark:main-ann}. Assume $b=b_1+b_2$ is of degree $n$. If the two terms, $b_1$ and $b_2$ are monomials in $S^{sk}$ and have a common factor $l$, i.e. $b_1=la_1$ and $b_2=la_2$, then $b=l(a_1+a_2)$ where $a_1$ and $a_2$ are of degree at most $n-1$. By  the induction assumption the proposition holds for the binomial $a_1+a_2$, i.e., $a_1+a_2 \in W_{n-1}$, hence we have 
\begin{equation*}
b=l(a_1+a_2)\in l(W)_{n-1}\subset (W)_n.
\end{equation*}
If the two terms, $b_1$ and $b_2$ do not have any common factor then with the same method as above we can rewrite the binomial $b$ by adding and subtracting a term $m$ of degree $n$, which has a common factor $m_1$ with $b_1$ and a common factor $m_2$ with $b_2$, and we will have

\begin{equation*}
b_1+b_2= b_1+m+b_2-m=m_1(c_1+m')+m_2(c_2-m''),
\end{equation*}
where $b_1=m_1c_1$, $m=m_1m'=m_2m''$ and $b_2=m_2c_2$. Since $c_1+m'$ and $c_2-m''$ are of degree at most $n-1$, the induction assumption yields
\begin{equation*}
b_1+b_2=m_1(c_1+m')+m_2(c_2-m'')\in(W)_n.
\end{equation*}
This completes the induction step and the proof of the proposition.
\end{proof}

\vspace{.2in}
\begin{cor} \label{cor:Pfaffian-main-cor-k}
For $1\leq k \leq n$ we have
\begin{equation*}
(W)_k=\mathrm{Ann}(Pf(X^{sk}))\cap S^{sk}_{k}
\end{equation*}

We also have $(W)_{n+1}=S^{sk}_{n+1}$.

\end{cor}

\begin{proof}
Using Equation \ref{eq:Pfaffian-main-easy-k-inclusion} we only need to show that 
\begin{equation*}
\mathrm{Ann}(Pf(X^{sk}))\cap S^{sk}_{k}\subset (W)_k
\end{equation*}  

By Remark \ref{remark:introIK} and Lemma \ref{lem:SkcirclePfaff} we have
\begin{equation*}
(\mathrm{Ann}(Pf(X^{sk})))_k=(\mathrm{Ann}(S_{n-k}\circ Pf(X^{sk})))_k=(\mathrm{Ann}(P_{2k}(X^{sk})))_k
\end{equation*}

Now if we label the $2k\times 2k$ Pfaffians of $X^{sk}$ by $f_1,...,f_s$ we have

\begin{equation*}
\mathrm{Ann}(P_{2k}(X^{sk})))_k=(\mathrm{Ann}<f_1,...,f_s>)_k=(\bigcap_{i=1}^{i=s}(\mathrm{Ann}(f_i))_k
\end{equation*}

Let $R^i$ denote the ring in the variables of $f_i$ and $W(i)$ the $f_i$ variables that are involved. By Proposition~\ref{prop:Pfaffian-main-prop-n} we have
\begin{equation*}
(W(i))_k=Ann(f_i)\cap S^i_k
\end{equation*}

So we have 

\begin{equation*}
\mathrm{Ann}(Pf(X^{sk}))\cap S^{sk}_{k}\subset (W)_k
\end{equation*}

To prove the second part, it is easy to see that every monomial of degree larger than $n$ will be unacceptable, of type (a) or (b), so in $W$, and we have $(W)_{n+1}=S^{sk}_{n+1}$.
\end{proof}

\vspace{.2in}

\begin{thm}\label{thm:Pfaffian-main-theorem}
Let $X^{sk}$ be a generic skew symmetric $2n\times 2n$ matrix. Then the apolar ideal $\mathrm{Ann}(Pf(X^{sk}))$ is the ideal $W$ and is generated in degree 2.
\end{thm}

\begin{proof}
This follows directly from Proposition \ref{prop:Pfaffian-main-prop-n} and Corollary \ref{cor:Pfaffian-main-cor-k}.
\end{proof}
\vspace{.2in}

\begin{cor} Let $X^{sk}$ be a $2n\times 2n$ generic skew symmetric matrix. We have
\begin{equation}\label{eq:Pfaffian-rank-equation}
2^{2n-2} \leq cr(Pf(X^{sk}))\leq 2^{n-1}
\end{equation}
\end{cor}

\begin{proof}
By the Ranestad-Schreyer Proposition, Corollary \ref{cor:dimension-degree-Pfaffian-Hilbert} and Theorem \ref{thm:Pfaffian-main-theorem} we have
\begin{equation*}
cr(Pf(X^{sk}))\geq \frac{1}{2} \dim(S^{sk}/\mathrm{Ann}(Pf(X^{sk})))=\frac{1}{2}(2^{2n-1})=2^{2n-2}.
\end{equation*}
The second inequality is true by Equation \ref{eq:Pedro}.
\end{proof}

\vspace{.2in}
\begin{remark}\label{remark:l-diff-pfaffian}
For $n\ge 5$ it can be easily seen that the lower bound for the cactus rank given by Corollary 4.12 is larger than $l_{diff}={{2n}\choose {2t_0}}$, where $t_0=\lfloor n/2\rfloor$.
\end{remark}
\vspace{.2in}
\begin{thm}\label{thm:Hafnian-main-theorem-cor}
Let $X^{s}$ be a generic symmetric $2n\times 2n$ matrix. Then the apolar ideal $\mathrm{Ann}(Hf(X^{s}))$ is generated in degree 2, and the inequality \ref{eq:Pfaffian-rank-equation} also holds for $(Hf(X^{s}))$.
\end{thm}

\begin{proof}
By the definition of the Hafnian, it is easy to see that none of the diagonal elements appear in $Hf(X^s)$, so for $1\leq i\leq 2n$ we have
\begin{equation*}
y_{ii}\circ Hf(X^s)=0
\end{equation*}

Hence without loss of generality we can restrict our discussion to the case where $X^s$ is a generic zero-diagonal symmetric matrix. By changing the Pfaffians to Hafnians and vice versa, the proof follows directly from the proofs that we have for the Pfaffian of a generic skew symmetric matrix.

\end{proof}


\section{Gr\"{o}bner bases}

In Section \ref{generic} we have shown that for $A$ a generic $n\times n$ matrix $\mathrm{Ann}(\det(A))=(\mathcal P_D+\mathcal U_D)$. In \cite{LS}, R. Laubenbacher and I. Swanson give a Gr\"{o}bner bases for the ideal of $2\times 2$ permanents of a matrix. In this section we first review their result (Theorem \ref{thm:generic-grobner-LS}) and then state our result for the ideal $\mathrm{Ann}(\det(A))$ and prove it independently (Theorem \ref{thm:generic-grobner-alternative-proof}).


\vspace{.2in}
\begin{defi}\label{def:diagonal order}(\cite{LS}, page 197)
Let $D=(d_{ij})$ be the matrix of the differential operators as defined in section \ref{intro}. A monomial order on the $d_{ij}$ is \emph{diagonal} if for any square submatrix of $D$, the leading term of the permanent (or of the determinant) of that submatrix is the product of the entries on the main diagonal. An example of such an order is the lexicographic order defined by:
\begin{equation*}
d_{ij}<d_{kl} \text{ if and only if } l>j \text{or } l=j \text{ and } k>i.
\end{equation*}
\end{defi}

Throughout this section we use a lexicographic diagonal ordering.

\vspace{.2in}

\begin{thm}\label{thm:generic-grobner-LS}(\cite{LS}, page 197)
The following collection $G$  of polynomials is a minimal reduced Gr\"{o}bner basis for  $\mathcal P_D$, with respect to any diagonal ordering:

(1) The subpermanents $d_{ij}d_{kl}+d_{kj}d_{il}$, $i<k$,$j<l$;

(2) $d_{i_1j_1}d_{i_1j_2}d_{i_2j_3}, i_1>i_2,j_1<j_2<j_3$;

(3) $d_{i_1j_1}d_{i_2j_2}d_{i_2j_3}, i_1>i_2,j_1<j_2<j_3$;

(4) $d_{i_1j_1}d_{i_2j_1}d_{i_3j_2}, i_1<i_2<i_3,  j_1>j_2$;

(5) $d_{i_1j_1}d_{i_2j_2}d_{i_3j_2}, i_1<i_2<i_3,  j_1>j_2$;

(6) $d_{i_1j_1}^{e_1}d_{i_2j_2}^{e_2}d_{i_3j_3}^{e_3}, i_1<i_2<i_3, j_2>j_3, e_1e_2e_3=2$.

\end{thm}
\vspace{.2in}

Monomials of type (2), (3), (4), (5) and (6) in the above theorem are in the ideal generated by all unacceptable monomials. 
\vspace{.2in}

\begin{thm}\label{thm:generic-grobner-alternative-proof}

The collection of unacceptable degree 2 monomials and $2\times 2$ subpermanents of $D$, form a Gr\"{o}bner basis for $\mathrm{Ann}(\det(A))$ with respect to any diagonal ordering.

\end{thm}



\begin{proof}We will denote $\mathcal U_D $ and $\mathcal P_D$ by $\mathcal U, \mathcal P$ respectively in the following, where $D$ is understood.\par
The elements of $(\mathcal U+\mathcal P)$ generate $\mathrm{Ann}(\det(A))$. Since $\mathcal U$ is a set of monomials, it is already  Gr\"{o}bner. We use Buchberger's algorithm to find a Gr\"{o}bner basis for $\mathcal P+\mathcal U$. We consider several cases:

a) Let $F$ and $G$ be distinct permanents of $D$. Let $F=a_{ik}a_{jl}+a_{il}a_{jk}$ and $G=a_{uz}a_{vw}+a_{uw}a_{vz}$ be two permanents in $\mathcal P$. 

\begin{center}  $ F= perm\left(%
\begin{array}{cc}
 
  a_{ik} & a_{il} \\
  a_{jk} & a_{jl}\\
  \end{array}%
\right)$.

\end{center}
and 
\begin{center}  $ G=perm \left(%
\begin{array}{cc}

  a_{uz} & a_{uw} \\
  a_{vz} & a_{vw}\\
  \end{array}%
\right)$.

\end{center}

Let $f_1=a_{ik}a_{jl}$ be the leading term of $F$, and $g_1=a_{uz}a_{vw}$ be the leading term of $G$ with respect to the given diagonal ordering. Denote the least common multiple of $f_1$ and $g_1$ by $h_{11}$. Let 

\begin{equation*}
S(F,G)=(h_{11}/f_1)F-(h_{11}/g_1)G=a_{uz}a_{vw}a_{il}a_{jk}-a_{ik}a_{jl}a_{uw}a_{vz}.
\end{equation*}

Now using the multivariate division algorithm, reduce all the $S(F,G)$ relative to the set of all permanents. When there is no common factor in the initial terms of $F$ and $G$ the reduction is zero, as one can use $F$ and $G$ again as we show. First we reduce $S(F,G)$ dividing by $F\in\mathcal P$, so we will have

\begin{equation*}
S(F,G)+a_{uw}a_{vz}(a_{ik}a_{jl}+a_{il}a_{jk})=a_{uz}a_{vw}a_{il}a_{jk}+a_{uw}a_{vz}a_{il}a_{jk}.
\end{equation*}
Then we reduce the result using $G$ this time, so we will have
\begin{equation*}
a_{uz}a_{vw}a_{il}a_{jk}+a_{uw}a_{vz}a_{il}a_{jk}-a_{il}a_{jk}(a_{uz}a_{vw}+a_{uw}a_{vz})=0.
\end{equation*}

So we have shown that for all pairs $F$, $G$ of distinct permanents of $D$, the $S$-polynomials $S(F,G)$ reduces to zero with respect to $\mathcal P$. 

b) Let $F=a_{ik}a_{jl}+a_{il}a_{jk}$ and $G=a_{ik}a_{jm}+a_{im}a_{jk}$ be two permanents so that their initial terms have a common factor. We have 

\begin{equation*}
S(F,G)=a_{il}a_{jk}a_{jm}-a_{im}a_{jk}a_{jl}\in \mathcal U.
\end{equation*}

c) Let $F=a_{im}a_{jn}+a_{in}a_{jm}$ be a permanent and $M=a_{tk}a_{tl}$ be an unacceptable monomial. We have

\begin{equation*}
S(F,M)=a_{tk}a_{tl}a_{jm}a_{in} \in \mathcal U.
\end{equation*}

d) Let $F=a_{il}a_{jm}+a_{im}a_{jl}$ be a permanent and $M=(a_{kn})^2$ be an unacceptable monomial. We have

\begin{equation*}
S(F,M)=a_{im}a_{jl}(a_{kn})^2 \in \mathcal U.
\end{equation*}

e) Let $F=a_{il}a_{jm}+a_{im}a_{jl}$ be a permanent and $M=(a_{il})^2$ be an unacceptable monomial which has a common factor with the initial term of $F$. We have

\begin{equation*}
S(F,M)=a_{il}a_{im}a_{jl} \in \mathcal U.
\end{equation*}

f) Let $F=a_{il}a_{jm}+a_{im}a_{jl}$ be a permanent and $M=a_{jn}a_{kn}$ be an unacceptable monomial. We have

\begin{equation*}
S(F,M)=a_{im}a_{jl}a_{jn}a_{kn} \in \mathcal U.
\end{equation*}

This exhausts all possibilities, so the generating set $\mathcal P+\mathcal U$ is itself a Gr\"{o}bner basis by Buchberger's algorithm.

\end{proof}

\subsection{Discussion of connected sum}

\vspace{.2in}
\begin{defi} (\cite{MS})
A polynomial $F$ in $r$ variables is a connected sum if we can write $F=F'+F''$ with $F'$ and $F''$ in $r'$ and $r''$ variables, where $r'+r''=r$.
\end{defi}
\vspace{.2in}
Let $A$  be a generic $2\times 2$  matrix, we can write the determinant $A$ is a  sum of two polynomials in complementary sets of variables.
\vspace{.2in}




\begin{prop}(Buczy{\'n}ska, Buczy{\'n}ski,Teitler (\cite{BBT})
If a form $F$ of degree $d$ is a connected sum, then the apolar ideal has a minimal generator in degree d. (The converse does not hold.) 
\end{prop}
In particular, since the generic determinant and permanent of size $n \ge 3$ have their annihilating ideals generated in degree 2, therefore they are not connected sums. This is also true for the Pfaffian of skew symmetric matrices and Hafnian of symmetric matrices of size $n\ge 6$.

\vspace{.2in}





\begin{thebibliography}{11}
\renewcommand{\baselinestretch}{1.7}


\bibitem[BBT]{BBT}
W. Buczy{\'n}ska, J. Buczy{\'n}ski, and Z. Teitler: \emph{Apolarity criteria for connected sum decomposability of polynomials}, in preparation (2012).

\bibitem[BBT2]{BBT2}
W. Buczy{\'n}ska, J. Buczy{\'n}ski, and Z. Teitler: \emph{Waring decompositions of monomials}, Journal of Algebra 378 (2013): 45-57. arXiv:1201.2922.


\bibitem[BC]{BC}
W. Bruns and A. Conca: \emph{Gr\"{o}bner bases and determinantal ideals},  Commutative algebra, singularities and computer algebra (J. Herzog et al., eds.), NATO Sci. Ser. II Math. Phys. Chem., 115, pp. 9--66, Kluwer, Dordrecht, 2003. 


\bibitem[BR]{BR}
A. Bernardi and K. Ranestad: \emph{On the cactus rank of cubic forms}, Journal of Symbolic Computation,10.1016/j.jsc.2012.08.001, arXiv:1110.2197.

\bibitem[CCG]{CCG}
E. Carlini, M. V. Catalisano and A. V. Geramita: \emph{The solution to Waring problem for monomials}, arXiv: 1110.0745v1(2011).

\bibitem[Ge]{Ge}
A. V. Geramita: \emph{Inverse Systems of Fat Points: Waring's Problem, Secant Varieties of Veronese Varieties and Parameter Spaces for Gorenstein Ideals}, Queen's Papers in Pure and Applied Mathematics, 102, pp. 3-104 (1996).



\bibitem[HT]{HT}
J. Herzog and N. Trung: \emph{Gr\"{o}bner bases and multiplicity of Determinantal and Pfaffian ideals}, Advances in Mathematics 96, 1-37 (1992).

\bibitem[IK]{IK}
A. Iarrobino and V. Kanev: \emph{Power Sums, Gorenstein Algebras, and Determinantal Varieties}, (1999), 345+xxvii p., Springer Lecture Notes in Mathematics \#1721.

\bibitem[IKO]{IKO}
M. Ishikawa, H. Kawamuko and S. Okada: \emph{A Pfaffian-Hafnian analogue of Borchardt's identity}, Electronic Journal of Combinatorics 12, Note 9 (2005).

\bibitem[LS]{LS}
R.C. Laubenbacher and I. Swanson: \emph{Permanental Ideals}, Journal of Symbolic Computation 30, 195-205 (2000).

\bibitem[LT]{LT}
J.M. Landsberg and Z. Teitler: \emph{On the ranks and border ranks of symmetric tensors}, Foundations of Computational Mathematics 10.3, 339-366 (2010).

\bibitem[MS]{MS}
D. Meyer and L. Smith: \emph{Poincare Duality Algebras, Macaulay's Dual Systems, and Steenrod Operations}, Cambridge Tracts in Mathematics, 167. Cambridge University Press, Cambridge, 2005. 


\bibitem[RS]{RS}
K. Ranestad and F.-O. Schreyer: \emph{On the rank of a symmetric form}, Journal of Algebra 346 (2011), 340-342.
arXiv:1104.3648 (2011).

\bibitem[Sh2]{Sh2}
M. Shafiei: \emph{Apolarity for determinants and permanents of generic symmetric matrices}, 
arXiv:1303.1860 (2013).

\bibitem[ST]{ST}
R. Stanley: \emph{Enumerative combinatorics}, Volume 1. Second edition. Cambridge Studies in Advanced Mathematics, 49. Cambridge University Press, Cambridge, 2012. 
\end{thebibliography}
\end{document}